\documentclass[a4paper,10pt]{article}

\usepackage{amsfonts, epsfig, amsmath, amssymb, color}
\usepackage{mathrsfs, amsthm}

\usepackage{mathrsfs}
\usepackage{pgfplots}
\pgfplotsset{width=7cm,compat=newest}
\usepackage{pgfplotstable}
\usetikzlibrary{calc}

\theoremstyle{plain}
\newtheorem{thm}{Theorem}[section]

\theoremstyle{definition}
\newtheorem{rem}[thm]{Remark}
\newtheorem{exa}[thm]{Example}
\newtheorem{dfn}{Definition}[section]

\newcommand{\be}{\begin{equation}}
\newcommand{\ee}{\end{equation}}
\newcommand{\ben}{\begin{equation*}}
\newcommand{\een}{\end{equation*}}

\newcommand{\ba}{\begin{equation}\begin{aligned}}
\newcommand{\ea}{\end{aligned}\end{equation}}
\newcommand{\ban}{\begin{equation*}\begin{aligned}}
\newcommand{\ean}{\end{aligned}\end{equation*}}

\renewcommand{\i}{\mathrm{i}}
\newcommand{\ex}{\mathrm{e}}
\newcommand{\di}{\mathrm{d}}
\newcommand{\eex}{\mathbf{e}}

\renewcommand{\phi}{\varphi}
\newcommand{\cP}{\mathcal{P}}
\newcommand{\cQ}{\mathcal{Q}}
\newcommand{\rF}{\mathscr{F}}
\newcommand{\bF}{\mathbb{F}}
\newcommand{\bI}{\mathbb{I}}
\newcommand{\bR}{\mathbb{R}}
\newcommand{\E}{\mathbf{E}}
\renewcommand{\P}{\mathbf{P}}

\begin{document}
 
\title{First Order Linear Marcus SPDEs}

\author{Lena--Susanne Hartmann\footnote{Institute of Mathematics, Friedrich Schiller University Jena, Ernst--Abbe--Platz 2,
07743 Jena, Germany;
 lena.hartmann@uni-jena.de} \ and Ilya Pavlyukevich\footnote{Institute of Mathematics, Friedrich Schiller University Jena, Ernst--Abbe--Platz 2,
07743 Jena, Germany; ilya.pavlyukevich@uni-jena.de}}

\maketitle

\begin{abstract}
 In this paper we solve a L\'evy driven linear stochastic first order partial differential equation (transport equation) understood in the canonical 
(Marcus) form.
The solution can be obtained with the help of the method
of stochastic characteristics. It has the same form as a solution of a deterministic PDE or a solution of a stochastic PDE driven by a Brownian motion
studied by Kunita (1984, 1997). 
\end{abstract}

\noindent 
\textbf{Keywords:} Canonical (Marcus) SDE; linear first order partial differential equation; transport equation; 
stochastic characteristics.

\noindent 
\textbf{AMS Subject Classification:} 35F10, 35R60, 60G51, 60H05 60H15

\section{Introduction}

The transport of a contaminant in an incompressible fluid can be described with the help of the first order linear partial differential equation 
(transport equation)
\ba
\label{e:tr}
\partial_t u(x,t)&=\nabla^T u(x,t)  a(x,t), \quad x\in\bR^d,\ t>0,\\
u(x,0)&=u_0(x),
\ea
where $u(x,t)$ is the concentration of the contaminant at time instant $t$ at the position $x\in\bR^d$, and $-a(x,t)$ is the flow instant velocity
at the position $x$ at time $t$, see e.g.\ Debnath \cite{debnath2005} and Van der Perk \cite{van2007soil}.

The equation \eqref{e:tr} can be solved with the help of the method of characteristics. 
Indeed, consider the so-called characteristics ordinary differential equation 
\ba
\label{e:ode}
\di \phi_{0,t}(x)&= -a(\phi_{0,t}(x),t)\, \di t,\quad \phi_{0,0}(x)=x,\quad t\geq 0.
\ea
Under usual smoothness and boundedness assumptions on the function $a$, the ODE \eqref{e:ode} has a unique global solution and
$t\mapsto\phi_{0,t}(\cdot)$ is a flow of diffemorphisms on $\bR^d$.
Then it is easy to check that the unique solution of the transport equation \eqref{e:tr} is given by the formula
\ba
\label{e:sold}
u(x,t)=u_0(\phi_{t,0}(x)),
\ea
where $x\mapsto \phi_{t,0}(x)=\phi^{-1}_{0,t}(x)$ is the inverse of the map 
$x\mapsto \phi_{0,t}(x)$, see e.g.\ Chapter 6 in Perthame \cite{perthame2006transport}.

Analogously one can show that the more general equation with the
``sediment deposition''  term $ub$ (if $b\leq 0$)
and
the ``sink/source'' term $c$ 
\ba
\label{e:trGen}
\partial_t u(x,t)&=\nabla^T u(x,t)  a(x,t) + u(x,t)b(x,t)+ c(x,t)  , \quad x\in\bR^d,\ t>0,\\
u(x,0)&=u_0(x),
\ea
has the solution 
\ba
\label{e:solGen}
u(x,t)=\ex^{\int_0^t b(\phi_{t,s}(x) ,s)\,\di s } u_0( \phi_{t,0}(x) ) 
+    \int_0^t  \ex^{\int_s^t b(\phi_{t,r}(x) ,r) \,\di r} c(  \phi_{t,s}(x) ,s  ) \,\di s,
\ea
see Section \ref{s:proof} for detail.

In a turbulent flow, the fluid velocity is given by a turbulent solution of Navier--Stokes equations
that is very difficult to analyze.
Instead, one can consider a plausible significant simplification, namely one can assume that the fluid velocity is represented as a sum
of the mean velocity field and a spatially dependent noise, see Man and Tsai \cite{ManTsai07}, Wang and Zheng \cite{WangZheng05}.
This yields the stochastic equation
\ban
\partial_t u(x,t)&=\nabla^T u(x,t)  a(x,t)+ \nabla^T u(x,t)  A(x,t) \circ \dot  W(t),\\
u(x,0)&=u_0(x). 
\ean
Here $W$ is a Brownian motion, and $\circ\, \dot W$ denotes the noise term in the Stratonovich sense.
Introducing the characteristics SDE in the Stratonovich form
\ban
\di \phi_{0,t}(x)&= a(\phi_{0,t}(x),t)\, \di t+ A(\phi_{0,t}(x),t)\circ \di W_t,\quad \phi_{0,0}(x)=x,
\ean
we obtain the same representation \eqref{e:sold} for the solution where $t\mapsto \phi_{t,0}(\cdot)$ is now the inverse flow of stochastic diffemorphisms.  
The general theory can be found in Kunita \cite{Kunita84,Kunita97} as well as in Chow \cite{chow2015stochastic},  Lototsky and Rozovsky \cite{lototsky2017stochastic}, and Flandoli \cite{flandoli2011interaction,Flandoli-2011}.

The goal of this paper is to show that a L\'evy driven first order equation can be also solved by the method of stochastic characteristics and hence to
extend the results by Kunita \cite{Kunita84,Kunita97} to the case of L\'evy driven equations.
Apart from the mathematical interest, this study is motivated by the questions appearing in the modelling of contaminant transport. Considering  
Le\'vy noise allows one to incorporate such effects as asymmetry of flow fluctuations and presence of extreme and large flow perturbations,
see, e.g.\ Kallianpur and Xiong \cite{kallianpur1994stochastic},
Birnir \cite[Section 1.7]{birnir2013kolmogorov}, Oh and Tsai \cite{OhTsai-10}, and Tsai and Huang \cite{tsai2019modeling}.

A general theory for L\'evy-driven stochastic partial differential equations can be found in Peszat and Zabczyk \cite{PesZab07}, Mandrekar and R{\"u}diger \cite{mandrekar2015stochastic}
(evolution equation approach), and Holden et al.\ \cite{holden2010stochastic} (white noise theory approach).
It should be noted that
in order to have the formulae \eqref{e:sold} and \eqref{e:solGen} valid in the case of a L\'evy-driven equation, 
we have to understand the stochastic term in the canonical (Marcus) form. 
This approach is different from above mentioned approaches, in particular from those by Proske \cite{Proske-04} who obtained a solution of a transport equation driven by a multiparameter 
L\'evy process 
within the frame of the white noise theory.

In the present paper we will assume that all the coefficients of the transport equation are sufficiently smooth. There is vast literature 
devoted to the analysis of (stochastic) transport equations with irregular coefficients in the countinous setting, see, e.g.\ 
DiPerna and Lions \cite{diperna1989ordinary},
Ambrosio \cite{Ambrosio-08},
Flandoli \cite{flandoli2011interaction,Flandoli-2011}, 
Fedrizzi and Flandoli \cite{fedrizzi2013noise}, 
Fedrizzi et al.\ \cite{fedrizzi2018class},
Catuogno and Olivera \cite{catuogno2013lp},
Olivera and Tudor \cite{olivera2015density},
Mollinedo and Olivera \cite{mollinedo2017well}.
To our best knowledge, L\'evy-driven transport equations with irregular coefficients are still to be studied.

\medskip

\noindent
\textbf{Notation.} By $C^n$ we denote $n$ times continuously differentiable functions and by 
$C^n_b$ we denote bounded $n$ times continuously differentiable functions with bounded derivatives. For a mapping 
$F\colon \bR^d\mapsto \bR^m$, $DF$ is the Jacobian (gradient) matrix, namely, 
\ban
DF=\begin{pmatrix}
    \frac{\partial F_1}{\partial x_1}& \cdots &\frac{\partial F_1}{\partial x_d}\\
    \cdots& \cdots& \cdots\\
    \frac{\partial F_m}{\partial x_1}& \cdots &\frac{\partial F_m}{\partial x_d}\\
   \end{pmatrix},
\ean
and in particular, for $F\colon \bR^d\mapsto \bR$, $DF=(\frac{\partial F}{\partial x_1}, \cdots ,\frac{\partial F}{\partial x_d})=\nabla^T F$.
The exponential function is denoted by $x\mapsto \ex^x$ whereas the 
exponential map in the sense of \eqref{e:expmap}, \eqref{e:expmap1}, \eqref{e:h2} and \eqref{e:h2} is denoted by the bold symbol $\eex$.

\medskip

\noindent
\textbf{Acknowledgements.} The authors acknowledge support
 by the DFG project PA 2123/4-1. The authors are grateful to the referees for their valuable comments.

\section{Setting and the main results}

On a filtered probability space $(\Omega,\rF,\bF,\P)$ that satisfies the usual hypotheses we consider an $m$-dimensional Brownian motion $W$ and an $m$-dimensional pure jump L\'evy process $Z$,
\ban
Z_t=\int_0^t\int_{\|z\|\leq 1} z\tilde N(\di z,\di s)+ \int_0^t\int_{\|z\|> 1} z  N(\di z,\di s),
\ean
where $N$ is a Poisson random measure with the intensity L\'evy measure $\nu$, $\nu(\{0\})=0$, $\int_{\bR^m}(|z|^2\wedge 1)\,\nu(\di z)<\infty$ and 
$\tilde N$ is the compensated Poisson random measure, $\tilde N(\di z,\di t)=N(\di z,\di t)-\bI(|z|\leq 1)\nu(\di z)\di t$.
For the exposition of the theory of L\'evy processes and L\'evy driven stochastic calculus we refer the reader to e.g.\ 
Sato \cite{Sato-99}, Kunita \cite{Kunita-04} and Applebaum \cite{Applebaum-09}.

Let
\ban
a &\colon\bR^d\to \bR^{d},&&&  b,c &\colon\bR^d\to \bR,\\
A,\alpha &\colon\bR^d\to \bR^{d\times m},&&& B,C,\beta,\sigma &\colon\bR^d\to \bR^{m},
\ean
be measurable functions.
For $u\in C^{1}(\bR^d,\bR)$, consider first order operators 
\ban
\cP u(x)& = \nabla^T u(x)  a(x) + u(x)b(x)+ c(x),\\
\mathcal R u(x)& = \nabla^T  u(x)  A(x) + u(x)B(x)+ C(x),\\
\mathcal Q u(x)& = \nabla^T u(x)  \alpha (x) + u(x)\beta(x)+ \sigma (x),\\
\ean
and the first order linear equation written in the compact differential form as
\ba
\label{e:SPDE}
\di u(t,x)&=\cP u(t,x)\,\di t+ \mathcal R u(t,x)\circ\di W_t + \mathcal Q u(t,x)\diamond\di Z_t,\\
u(0,x)&=u_0(x),\quad x\in\bR^d,\ t\in[0,T],
\ea
with some initial condition $u_0\colon\bR^d\to\bR$.
More precisely, this equation is understood as the integral equation as follows:
\ba
\label{e:SPDE-2}
u(t,x)&=u_0(x)+\int_0^t\Big[ \nabla^T u(s,x)a(x)  + u(s,x)b(x)+ c(x)\Big]\,\di s\\ 
&+   \int_0^t\Big[  \nabla^T u_{x}(s,x)A(x)  +  u(s,x) B (x)  + C (x)\Big]\circ\di W_s
\\ 
&+ \int_0^t\int_{\|z\|\leq 1} \Big(\eex^{\cQ z } u(s-,x) - u(s-,x) \Big) \tilde N(\di z,\di s)\\
&+ \int_0^t\int_{\|z\|\leq 1} \Big(\eex^{\cQ z} u(s-,x) - u(s-,x) - \mathcal Q u(s-,x) z \Big)\, \nu(\di z)\,\di s\\
&+  \int_0^t\int_{\|z\|> 1} \Big(\eex^{\cQ z } u(s-,x) - u(s-,x) \Big) \tilde N(\di z,\di s),
\ea
where for each $z\in \bR^m$ the mapping
\ba
\label{e:expmap}
u(\cdot )\mapsto \eex^{\cQ z} u(\cdot)
\ea
is defined with the help of the solution of the first order linear time autonomous partial differential equation
\ba
\label{e:g}
\partial_r g(r,x)&= \nabla^T g(r,x)  \alpha (x)z + g(r,x)\beta(x)z+ \sigma (x)z  ,\quad  r\in[0,1],\\
g(0,x)&=u(x),
\ea
and
\ba
\label{e:expmap1}
\eex^{\cQ z}u(x): =g(1;x,z).
\ea
It is possible to write down the solution of the equation \eqref{e:g} explicitly. We will do this in Section \ref{s:proof}.

The term $\circ\,\di W$ denotes the Stratonovich integral w.r.t.\ the Brownian motion.

\begin{dfn}
A random field $u\colon \bR^d\times \bR_+\times \Omega\to\bR$ is a solution of equation \eqref{e:SPDE}
if $u$ is a c\`adl\`ag $C^2$-semimartingale and \eqref{e:SPDE-2} is satisfied a.s.
\end{dfn}

The goal of this paper is to solve \eqref{e:SPDE} with the help of the method of characteristics. The main result is as follows.
Assume that the functions satisfy the following conditions.

\smallskip 

\noindent
\textbf{H}$_{\text{coeff}}$
\ba
\label{e:BDab}
a&\in C^3_b(\bR^d,\bR^{d})\quad \text{and}\quad b,c\in C^3_b(\bR^d,\bR) \\
A,\alpha &\in C^4_b(\bR^d,\bR^{d\times m})  \quad \text{and}\quad  B,C,\beta,\sigma \in C^4_b(\bR^d,\bR^{m}). 
\ea
Consider a $(d+2)$-dimensional  system of Marcus SDEs (characteristics equations)
\begin{align}
&\begin{aligned}
\label{e:phi}
\phi_{0,t}(x) =x -  \int_0^t a (\phi_{0,r}(x))\, \di r 
&- \int_0^t A  (\phi_{0,r}(x))\circ \di W_r \\
&-  \int_0^t \alpha(\phi_{0,r}(x)) \diamond \di Z_r,
\end{aligned}\\
&\begin{aligned}
\label{e:xi}
\xi_{0,t}(x,\xi_0)= \xi_0- \int_0^t \xi_{0,r}(x,\xi_0) b( \phi_{0,r}(x) )\,\di r
&-\int_0^t  \xi_{0,r}(x,\xi_0) B(\phi_{0,r}(x))\circ \di W_r\\
&-\int_0^t\xi_{0,r}(x,\xi_0)\beta(\phi_{0,r}(x))\diamond \di Z_r,
\end{aligned} \\
&\begin{aligned}
\label{e:zeta}
\zeta_{0,t}(x,\xi_0,\zeta_0)=\zeta_0
-\int_0^t \xi_{0,r}(x,\xi_0)c(\phi_{0,r}(x))\,\di r 
&- \int_0^t \xi_{0,r}(x,\xi_0)C(\phi_{0,r}(x))\circ\di W_r\\
&- \int_0^t \xi_{0,r}(x,\xi_0)\sigma (\phi_{0,r}(x))\diamond \di Z_r.
\end{aligned} 
\end{align}
Under the assumptions \textbf{H}$_{\text{coeff}}$,
there exists a unique strong solution to the SDE \eqref{e:phi}, \eqref{e:xi}, \eqref{e:zeta}. 
Furthermore the associated solution flow 
is a $C^{2}$-flow of diffeomorphisms of $\bR^{d+2}$, see Applebaum \cite[Theorems 6.10.5 and 6.10.10]{Applebaum-09}.

Let $(\phi_{t,0},\xi_{t,0},\zeta_{t,0})$ be the inverse flow of $(\phi_{0,t},\xi_{0,t},\zeta_{0,t})$.

\begin{thm}
\label{t:main}
Let assumptions \emph{\textbf{H}$_{\text{coeff}}$} hold true and let
$u_0\in C^2_b(\bR^d,\bR)$. Then the function 
\ba
\label{e:sol}
u(t,x):=\xi_{t,0}(x,1)u_0(\phi_{t,0}(x))+\zeta_{t,0}(x,1,0)
\ea
is the unique solution of \eqref{e:SPDE}.  
\end{thm}

The intuition behind the formula \eqref{e:sol} is quite clear. In the deterministic case, $A=\alpha=0$, $B=C=\beta=\sigma=0$, the formula
\eqref{e:sol} is exactly the formula \eqref{e:solGen} and is well known. 
In the continuous case of Gaussian noise, $\alpha=0$, $B=\beta=0$, the formula \eqref{e:sol} was derived by Kunita, see 
\cite{Kunita84,Kunita97}. The important feature of the equation \eqref{e:SPDE} in the continuous case is that the stochastic integrals 
have to be considered in the Stratonovich sense. Informally this can be justified by the following consideration. It is well-known that the Stratonovich
stochastic integral can be approximated pathwise by Stieltjes integrals w.r.t.\ to approximations of a Brownian motion by random continuous 
functions of bounded variation, e.g.\ by polygonal lines (the so-called, Wong--Zakai approximations). Each of 
these approximations can be treated path-wise as a deterministic first order PDE that has a solution \eqref{e:sol} (or \eqref{e:solGen}), 
and hence the limit
should have the same form. In the case of jump noise, the role of the Stratonovich
stochastic differential equations is played by the Marcus (canonical) stochastic differential equations. These equations 
can be also seen as a limit of continuous Wong--Zakai approximations and enjoy the Newton--Leibniz change of variables formula 
of conventional calculus. Hence it is intuitively clear that the equation  \eqref{e:SPDE} has to be considered as a Marcus 
equation.

\begin{exa}[transport equation] Consider the transport equation
\ba
\label{e:TE}
u(t,x)=u_0(x)&+ \int_0^t  \nabla^T u(r,x) a(x)\,\di r+ \int_0^t \nabla^T u(r,x)A(x) \circ\di W_r\\
&+\int_0^t \nabla^T  u(r,x)\alpha (x)\diamond\di Z_r.
\ea
In this case, the characteristics equation is a $d$-dimensional Marcus SDE
\ban
\phi_{0,t}(x) &=x - \int_0^t a (\phi_{0,r}(x))\, \di r 
-\int_0^t A(\phi_{0,r}(x))\circ \di W_r 
-\int_0^t \alpha(\phi_{0,r}(x))\diamond \di Z_r.
\ean
Then the solution has the form
\ban
u(t,x)=u_0(\phi_{t,0}(x)).
\ean
\end{exa}

\begin{exa}[explicit one-dimensional solution]
In dimension $m=d=1$ if $a(x)=A(x)=\alpha(x)$, the equation \eqref{e:TE} can be solved explicitly with the help of the It\^o formula for Marcus SDEs.
Indeed, let $Z$ be a general (not necessarily pure jump) one-dimensional L\'evy process, and consider the equation
\ba
\label{e:TEe}
u(t,x)=u_0(x)+ \int_0^t  \nabla u(r,x) \alpha (x)\diamond\di Z_r.
\ea
Assume that $\alpha(x)>0$ and denote 
\ban
H(x)=\int_0^x \frac{\di y}{\alpha(y)},\quad x\in\bR.
\ean
Then the characteristics equation
\ban
\phi_{0,t}(x) &=x - \int_0^t \alpha(\phi_{0,r}(x)) \diamond \di Z_r
\ean
has the solution
\ban
\phi_{0,t}(x)=H^{-1}(H(x)-Z_t)
\ean
and the inverse flow $\phi_{t,0}$ can be found by a straightforward calculation as 
\ban
\phi_{t,0}(x)=H^{-1}(H(x)+Z_t).
\ean
Let us show that 
\ba
\label{e:expl}
u(t,x):=u_0(H^{-1}(H(x)+Z_t))
\ea
satisfies \eqref{e:TEe}. Indeed, $u(0,x)=u_0(x)$, and
\ban
\nabla u(t,x)=\nabla u_0( \phi_{t,0}(x)   ) \nabla \phi_{t,0}(x) = \nabla u_0( \phi_{t,0}(x)   ) \frac{\alpha(H(x)+Z_t)}{\alpha(x)}.
\ean
On the other hand, the It\^o formula for Marcus SDEs (see, e.g.\ Section 4 in Kurtz et al.\ \cite{KurtzPP-95}) yields
\ban
u(t,x)=u_0(H^{-1}(H(x)+Z_t))&=u_0(x)+\int_0^t \frac{\partial}{\partial Z} u_0(H^{-1}(H(x)+Z_r))\diamond \di Z_r\\
&=u_0(x)+\int_0^t  \nabla u_0(  \phi_{0,r}(x) ) \alpha(H(x)+Z_r)  \diamond \di Z_r\\
&=u_0(x)+\int_0^t  \nabla u(r,x)\alpha(x)   \diamond \di Z_r.
\ean
\end{exa}

\begin{exa}
In this example we apply formula \eqref{e:expl} to the first order equation
\ba
\label{e:ess}
\partial_t u(t,x)=\nabla u(t,x)\sqrt{x^2+1} \diamond \di Z_t,\quad u(0,x)=u_0(x).
\ea
Note that although $\alpha(x)=\sqrt{x^2+1}$ is not bounded, the formula \eqref{e:sol} still holds true. Indeed, in this case, 
\ban
H(x)&=\operatorname{arcsinh}(x)=\ln(x+\sqrt{x^2+1})\quad \text{and}\quad \\
H^{-1}(x)&=\sinh(x)=\frac{\ex^x-\ex^{-x}}{2},\quad x\in\bR.
\ean
Hence, \eqref{e:ess} has the explicit solution
\ban
u(t,x)=u_0\Big( \sinh( \operatorname{arcsinh}(x) +Z_t   )  \Big),\quad t\geq 0,\ x\in\bR.
\ean
Sample paths of a symmetric $\alpha$-stable L\'evy process $Z$ and the solution $u$
are presented on Fig.\ \ref{f:fig1}.

\begin{figure}
\begin{center}
 \includegraphics[width=0.9\textwidth]{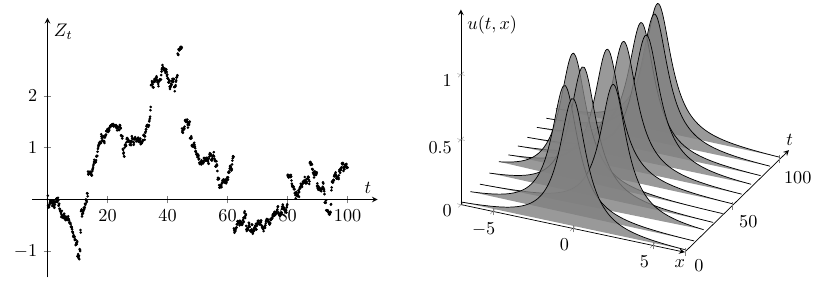}
\end{center}
\caption{
A sample path of a symmetric $\alpha$-stable L\'evy process
$Z$ with $\E\ex^{\i \lambda Z_1}= \ex^{-0.1|\lambda|^\alpha}$ for $\alpha=1.75$ (left); the solution $u(t,x)$
of equation \eqref{e:ess} with the initial condition $u_0(x)=1/(1+x^2)$ sampled at $t=0, 10, 20,\dots, 100$ (right).
\label{f:fig1}}
\end{figure}
\end{exa}

The paper is organized as follows. Since the Marcus integrals generalize the Stratonovich integrals in the jump setting, we will inspect 
Kunita's arguments in the continuous case and adapt them to the case of L\'evy noise. One of the central auxiliary results will be the 
generalized It\^o formula for Marcus SDEs (It\^o--Wentzell formula).

\section{Canonical (Marcus) SDEs\label{s:marcus}}

For the proof of the main theorem we have to recall several facts about canonical SDEs in the semimartingale setting.
Canonical SDEs were introduced by Marcus in \cite{Marcus-78}. They were studied by Marcus \cite{Marcus-81}, Fujiwara \cite{Fujiwara91}, Kurtz et al.\ \cite{KurtzPP-95} in the 
semimartingale setting and 
by Fujiwara and Kunita \cite{FujiwaraK-99a,FujiwaraK-99b} and Kunita \cite{Kunita-04} in the setting of semimartingales with a spatial parameter.

Let $(\Omega,\rF,\bF,\P)$ be a filtered probability space which satisfies the usual hypotheses and let $W$ and $N$ be an 
independent $m$-dimensional
Brownian motion, and a Poisson random measure on $\bR^m\backslash\{0\}$
with the compensator $\nu(\di z)\,\di t$ as defined in the previous section.

For $d\geq 1$, let 
\ban
f(x,r,\omega)&\colon \bR^d\times\bR_+\times \Omega\to \bR^d ,\\
F(x,r,\omega)&\colon \bR^d\times\bR_+\times \Omega\to \bR^{d\times m} ,\\
\phi(x,r,z,\omega)&\colon \bR^d\times\bR_+\times \bR^m\times\Omega\to \bR^{d} 
\ean
be predictable processes with parameters $x$ and $(x,z)$ respectively. In what follows, we will often 
omit the dependence on $\omega\in\Omega$.

We make the following assumptions. There are constants $\delta\in (0,1)$, $L>0$, $K>0$ such that

\noindent
\textbf{H}$_f$: 
\ba
&f(\cdot, r)\in C^{2+\delta}(\bR^d,\bR^d),\\ 
&\sup_x\frac{| f(x, r)|}{1+|x|}\leq K,\\ 
&\| \partial^\alpha f^i(\cdot,r) \|\leq K,\quad 1\leq i\leq d,\quad |\alpha|=1,2,\\
&\|\partial^\alpha  f^i(x, r)- \partial^\alpha   f^i(y, r)\|\leq  L\|x-y\|^\delta,\quad 1\leq i\leq d,\quad |\alpha|=2.  
\ea
\textbf{H}$_F$:
\ba
&F(\cdot, r)\in C^{3+\delta}(\bR^d,\bR^{d\times m}),\\ 
&\sup_x\frac{| F^i_j(x, r)|}{1+|x|}\leq K,\quad 1\leq i\leq d,\ 1\leq j\leq m,\\
&\| \partial^\alpha F^i_j(\cdot,r) \|\leq K,\quad 1\leq i\leq d,\ 1\leq j\leq m,\quad |\alpha|=1,2,\\
&\|\partial^\alpha F^i_j(\cdot,r) F^k_l(\cdot, r)\|\leq K, \quad 1\leq i,k\leq d,\ 1\leq j,l\leq m ,\quad |\alpha|=2,3, \\  
&\|\partial^\alpha  F^i_j(x, r)-\partial^\alpha F^i_j(y, r)\|\leq L\|x-y\|^\delta, \quad 1\leq i\leq d,\ 1\leq j\leq m,\quad |\alpha|=3.
\ea
\textbf{H}$_\phi$: there are non-negative functions $K_0(z)$, $K_1(z)$ and $L_1(z)$ and $L_2(z)$, such that
\ba
\int_{\|z\|\leq 1}\Big(K_0(z)^2+K_1(z)^2+ L_1(z)+L_2(z)^2   \Big)\,\nu(\di z)<\infty
\ea
and
\ban
&\phi(\cdot,r,z)\in C^{2+\delta}(\bR^d,\bR^{d\times m}),\\
&\sup_x\frac{|\phi^i_j(\cdot, r,z)|}{1+|x|}\leq K_0(z),\quad 1\leq i\leq d,\ 1\leq j\leq m ,\\
&\|\partial^\alpha \phi^i_j(\cdot,r,z)\|\leq  K_1(z), \quad 1\leq i,j\leq d,\quad 1\leq |\alpha|\leq 2,  \\
&|\partial^\alpha  \phi^i_j(x,r,z) \phi^k_l(x,r,z)-\partial^\alpha  \phi^i_j(y,r,z)\phi^k_l(y,r,z)|\leq  L_1(z)\|x-y\|,\\
&\hspace{.4\textwidth}  1\leq i,k\leq d,\ 1\leq j,l\leq m ,\quad |\alpha|=1,\\
&|\partial^\alpha  \phi^i_j(x,r,z)-\partial^\alpha  \phi^i_j(y,r,z)|\leq  L_2(z)\|x-y\|^\delta,\\
& \hspace{.4\textwidth} 1\leq i\leq d,\ 1\leq j\leq m,\quad |\alpha|=2.
\ean
Consider a semimartingale $\Phi$ with parameter $x$ given by 
\ba
\label{e:Phi}
\Phi(x,t)&=\int_0^t  f(x,r)\,\di r 
+ \int_0^t F(x,r)\,\di W_r\\
&+ \int_0^t\int_{\|z\|\leq 1}   \phi(x, r, z) \, \tilde N(\di z,\di r)
+ \int_0^t\int_{\|z\|> 1} \phi(x, r, z) \,  N(\di z,\di r) .
\ea
We want to give sense to the following Marcus SDE with the generator $\Phi(\cdot,\cdot)$, which we formally write as
\ba
\label{e:SDE}
X_{t}(x)&=x+\int_0^t \Phi( X_{r}(x),  \diamond\,\di r ).
\ea
This writing has the following meaning:

For each $r\geq 0$, $x\in\bR^d$, $z\in\bR^m$ and $\omega\in\Omega$ consider an ODE (the Marcus ODE) for the function $h=h(u)=h(u;x,r,z)$
\ba
\label{e:h1}
\begin{cases}
&\displaystyle\frac{\di }{\di u}h(u)=\phi(h(u),r,z), \quad u\in[0,1],\\
&\displaystyle h(0)=x.
\end{cases}
\ea
Under Assumption \textbf{H}$_\phi$ there is a global solution $h$ for $u\in \bR$ which we denote by
\ba
\label{e:h2}
\eex^{u\phi(\cdot;r,z)}(x):=h(u;x,r,z),\quad u\in\bR.
\ea
and in particular we define the exponential mapping
\ba
\label{e:h3}
\eex^{\phi(\cdot;r,z)}(x):=h(1; x,r, z).
\ea
For almost all $r$, $\omega$ and $z$ the mapping $x\mapsto \eex^{\phi(\cdot;r,z)}(x)$ is a $C^2$-diffeomorphism.

Consider the following functions which are well-defined due to Assumptions \textbf{H}$_f$, \textbf{H}$_F$ and \textbf{H}$_\phi$:
\ban
\hat f(x,r)&=f(x,r)+ \frac12 \sum_{j=1}^d DF_j(x,r)F_j(x,r) \\
&\qquad + \int_{\|z\|\leq 1} \Big(\eex^{\phi(\cdot,r,z)}(x) -x- \phi(x,r, z) \Big)  \, \nu(\di z),\\
\hat F(x,r)&= F(x,r),\\
\hat \phi(x,r,z)&=\eex^{\phi(\cdot, r, z)}(x) -x. 
\ean
Consider an It\^o semimartingale with parameter
\ban
\hat\Phi(x,t)&=\int_0^t \hat f(x,r)\,\di r 
+ \int_0^t \hat F(x,r)\,\di W_r\\
&+ \int_0^t\int_{\|z\|\leq 1} \hat \phi(x, r, z) \, \tilde N(\di z,\di r)
+ \int_0^t\int_{\|z\|> 1}  \hat \phi(x, r, z) \,  N(\di z,\di r) .
\ean
For an adapted c\`adl\`ag process $\eta=(\eta_t)_{t\in[0,T]}$ we set
\ban
\int_0^t \hat \Phi( \eta_{r-}, \di r  )&=
\int_0^t  \hat f(\eta_{r-},r)\,\di r 
+\int_0^t  \hat F(\eta_{r-},r)\, \di W_r \\
&+ \int_0^t\int_{\|z\|\leq 1} \hat \phi(\eta_{r-}, r, z) \, \tilde N(\di z,\di r)
+ \int_0^t\int_{\|z\|> 1}  \hat \phi(\eta_{r-}, r, z) \,  N(\di z,\di r)
\ean
and consider the It\^o SDE
\ba
\label{e:SDEIto}
X_{t}(x)&=x+\int_0^t \hat \Phi( X_{r-}(x),  \,\di r ).
\ea
Then for any $x$, there is a unique strong solution $X$ of \eqref{e:SDEIto}, see Kunita \cite[Theorem 3.1, Section 3.1]{Kunita-04}.
There is a modification of $X$ such that for each $x\in\bR^d$ the mapping $t\mapsto X_t(x)$ is c\`adl\`ag, and for each $t\in[0,T]$ the 
mappings $x\mapsto X_t(x)$ are $C^2$-diffemorphisms, see Kunita \cite[Section 3.5]{Kunita-04} and \cite[Section 4]{kunita1996stochastic}.

The functions $f$, $F$ and $\phi$ given, we will understand the canonical (Marcus) SDE \eqref{e:SDE} as the It\^o SDE  \eqref{e:SDEIto}.

\section{Generalized It\^o formula for canonical SDEs}

For the proof of formula \eqref{e:sol} we will need the so-called generalized It\^o formula for solutions of canonical SDEs. 
Let us first recall the conventional formula for the canonical Marcus SDEs which formally reminds of a conventional Newton--Leibniz formula
known in calculus. 

\begin{thm}[It\^o's formula for solutions of canonical SDEs] Let $X$ be the solution of the SDE \eqref{e:SDE} and 
let $\Theta\in C^2(\bR^d,\bR)$. Then
\ba
\label{e:itoM}
\Theta(X_t)&=\Theta(x)+\int_0^t \nabla^T\Theta\,\Phi(X_r,\diamond\, \di r),
\ea
where the canonical integral in the r.h.s.\ of \eqref{e:itoM} equals
\ban
\int_0^t &\nabla^T\Theta(X_{r-}) f(X_{r-},r)\,\di r
+ \int_0^t \nabla^T\Theta (X_{r-}) F(X_{r-},r)\circ \di W_r\\
&+\int_0^t\int_{\|z\|\leq 1} \Big(\Theta(\eex^{\phi(\cdot,r,z)}(X_{r-}) )-  \Theta(X_{r-})\Big) \tilde N(\di z,\di r)\\
&+\int_0^t\int_{\|z\|\leq 1} \Big(\Theta(\eex^{\phi(\cdot,r,z)}(X_{r-}) )-  \Theta(X_{r-}) -\phi(X_{r-},r,z)\Big) \nu(\di z)\,\di r\\
&+\int_0^t\int_{\|z\|> 1} \Big(\Theta(\eex^{\phi(\cdot,r,z)}(X_{r-}) )-  \Theta(X_{r-})\Big) N(\di z,\di r).
\ean
\end{thm}
\begin{rem}
Note that the process $\Theta(X)$ has the jumps
\ban
\Theta(X_{r})-\Theta(X_{r-})&=
\Theta(\eex^{\phi(\cdot,r,z)}(X_{r-}) )-  \Theta(X_{r-})\\
&=\int_0^1 \nabla^T\Theta( \eex^{u \phi(\cdot,r,z)}(X_{r-}) )\phi(\eex^{u \phi(\cdot,r,z)}(X_{r-}),r,z )\,\di u,
\ean
which justifies the formal writing \eqref{e:itoM}.
\end{rem}

Now we prove the generalized It\^o formula for a pair of canonical SDEs driven by the same Brownian motion and the 
Poisson random measure.

\begin{thm}[generalized It\^o formula for canonical SDEs] 
\label{t:Ito}
Consider solutions of canonical SDEs with generators $\Phi$ and $\Psi$
such that the functions $f,F,\phi$ and $g,G,\psi$ satisfy assumptions of the Section \ref{s:marcus} respectively,
\begin{align}
\label{e:X}
X_{t}(x)&=x+\int_0^t \Phi(X_r(x),\diamond\, \di r),\\
\label{e:Y}
Y_{t}&=y+\int_0^t \Psi(Y_r,\diamond\, \di r).
\end{align}
Then the following formula holds true
\ba
\label{e:ito-gen}
X_{t}(Y_t)
&=y+\int_0^t \Phi(X_{r}(Y_r),\diamond\, \di r)+\int_0^t D X\Psi(Y_r,\diamond\, \di r),
\ea
where the latter integrals are understood as
\ba
\label{e:Phidr}
&\int_0^t \Phi(X_{r}(Y_r),\diamond\, \di r)
= \int_0^t f(X_{r-}(Y_{r-}),r)\,\di r + \int_0^t F(X_{r-}(Y_{r-}),r)\, \di W_r \\
&+\frac12  \sum_{j=1}^m\int_0^t D F_j(X_{r-}(Y_{r-}),r)F_j(X_{r-}(Y_{r-}),r)\, \di r \\
&+ \int_0^t\int_{\|z\|\leq 1} \Big(  \ex^{\phi(\cdot, r,z)}(X_{r-}(Y_{r-}))-X_{r-}(Y_{r-})\Big)\tilde N(\di z,\di r) \\
&+\int_0^t\int_{\|z\|\leq 1} \Big(  \ex^{\phi(\cdot, r,z)}(X_{r-}(Y_{r-}))-X_{r-}(Y_{r-})- \phi(X_{r-}(Y_{r-}),r,z)\Big)\,\nu(\di z)\,\di r\\
&+ \int_0^t\int_{\|z\|> 1} \Big(  \ex^{\phi(\cdot, r,z)}(X_{r-}(Y_{r-}))-X_{r-}(Y_{r-})\Big)  N(\di z,\di r) \\
\ea
and 
\ba
\label{e:DPhidr}
\int_0^t & D X\Psi(Y_r,\diamond\, \di r)
= \int_0^t DX_{r-}(Y_{r-})  g(Y_{r-},r)\,\di r \\
&+\sum_{j=1}^m \frac12 \int_0^t  DX_{r-}(Y_{r-})D G_j(Y_{r-},r)  G_j(Y_{r-},r)\,\di r   \\
&+\frac12 \sum_{j,k=1}^m
\int_0^t (D(DX_{r-})_j)_k(Y_{r-}) G_j(Y_{r-},r)G_k(Y_{r-},r)\,\di r\\
&+\frac12  \sum_{j=1}^m \int_0^t  DX_{r-}(Y_{r-}) D F_j(X_{r-}(Y_{r-}),r) G_j(Y_{r-},r)\,\di r\\
&+ \int_0^t D X_{r-}(Y_{r-}) G(Y_{r-},r)\,\di W_r\\
&+\int_0^t\int_{\|z\|\leq 1}       
\Big[\ex^{\phi (\cdot,r, z)}\Big(X_{r-}\Big(\ex^{\psi(\cdot,r, z)}(Y_{r-})\Big)  
-  \ex^{\phi (\cdot,r, z)}(X_{r-}(Y_{r-}))   \Big]
\,  \tilde N(\di z,\di r)\\
&+\int_0^t\int_{\|z\|\leq 1}      
\Big[ 
\eex^{\phi (\cdot,r, z)}\Big(X_{r-}\Big(\eex^{\psi(\cdot,r, z)}(Y_{r-})\Big)\Big)-\eex^{\phi (\cdot,r, z)}(X_{r-}(Y_{r-}))  \\
&\hspace{0.45\textwidth}-  D X_{r-}(Y_{r-} )  \psi(Y_{r-},r, z) \Big] \, \nu(\di z)\,\di r \\
&+\int_0^t\int_{\|z\|>1}       
\Big[\ex^{\phi (\cdot,r, z)}\Big(X_{r-}\Big(\ex^{\psi(\cdot,r, z)}(Y_{r-})\Big)  
-  \ex^{\phi (\cdot,r, z)}(X_{r-}(Y_{r-}))   \Big]\, N(\di z,\di r).
\ea
\end{thm}

This Theorem will be proven in Section~\ref{s:gIf}.
We will also need the following formula which is proved analogously, see Section~\ref{s:ito2proof}.

\begin{thm} 
\label{t:Ito2}
Let $\Phi$ be a one-dimensional semimartingale given by \eqref{e:Phi} with a $d$-dimensional parameter $x$ and let $Y$ be a solution of the 
$d$-dimensional  canonical SDE \eqref{e:Y}. Then
\ba
\label{e:Phinabla}
\Phi(Y_t,t)=\int_0^t \Phi(Y_{r-}, \di r)+ \int_0^t  \nabla^T \Phi(Y_r,\diamond\, \di r),
\ea
where
\ba
\label{e:Phinabla1}
 &\int_0^t \Phi(Y_{r-}, \di r)= \int_0^t  f(Y_{r-},r)\,\di r 
+ \int_0^t F(Y_{r-},r)\,\di W_r\\
&+ \int_0^t\int_{\|z\|\leq 1}   \phi(Y_{r-}, r, z) \, \tilde N(\di z,\di r)
+ \int_0^t\int_{\|z\|> 1} \phi(Y_{r-}, r, z) \,  N(\di z,\di r) 
\ea
and
\ba
\label{e:Phinabla2}
&\int_0^t  \nabla^T \Phi(Y_r,\diamond\, \di r)=
\int_0^t  \nabla^T \Phi(Y_{r-},r)g(Y_{r-},r)\,\di r\\ 
&+ \int_0^t  \nabla^T \Phi(Y_{r-},r)G(Y_{r-},r)\circ \di   W_r \\
&+ \int_0^t  \int_{|z|\leq 1}  \int_0^1 \Big(\nabla^T \Phi(\ex^{\psi(\cdot,r, z)}(Y_{r-}),r-)+\nabla^T \phi(\ex^{\psi(\cdot,r, z)}(Y_{r-}),r,z))\Big)\times\\
&\hspace{.5\textwidth}\times\psi(\ex^{\psi(\cdot,r, z)}(Y_{r-}),r,z)\,\di u\, \tilde N(\di z,\di r)\\
&+
\int_0^t  \int_{|z|\leq 1} \Big[  \int_0^1 \Big(\nabla^T \Phi(\ex^{\psi(\cdot,r, z)}(Y_{r-}),r-)+\nabla^T \phi(\ex^{\psi(\cdot,r, z)}(Y_{r-}),r,z))\Big)
\times\\
&\hspace{.17\textwidth}\times \psi(\ex^{\psi(\cdot,r, z)}(Y_{r-}),r,z)\,\di u
- \nabla^T \Phi(Y_{r-},r- )\psi(Y_{r-},r,z)\Big]\, \nu(\di z)\,\di r\\
&+ \int_0^t  \int_{|z|> 1}  \int_0^1 \Big(\nabla^T \Phi(\ex^{\psi(\cdot,r, z)}(Y_{r-}),r-)+\nabla^T \phi(\ex^{\psi(\cdot,r, z)}(Y_{r-}),r,z))\Big)\times\\
&\hspace{.5\textwidth}\times\psi(\ex^{\psi(\cdot,r, z)}(Y_{r-}),r,z)\,\di u\,  N(\di z,\di r).
\ea
\end{thm}
 
\section{Equations for the inverse flows of the canonical SDEs}

We know from Kunita \cite{kunita1996stochastic,Kunita-04} that the 
solution $x\mapsto X_t(x)$, $t\geq 0$, maps $\bR^d$ onto itself diffeomorfically, and there exists a modification such that 
$X_{s,t}:=X_t\circ X^{-1}_s$ defines the stochastic flow of diffeomorphisms. Denote by $DX_t(x)$ its gradient (Jacobi) matrix,
that satisfies the so-called
variational SDE and is a (right) stochastic exponent. Let $(DX_t(x))^{-1}$ be its matrix inverse.

Consider the inverse flow $X_{t,0}:=X_{0,t}^{-1}$, $t\geq 0$. We show that the inverse flow satisfies the following formula.  
 
\begin{thm}
The inverse flow $t\mapsto X_{t,0}$, $t\geq 0$ satisfies the canonical SDE
\ba
\label{e:inveM}
X_{t,0}(x)=x-\int_0^t \Big(D X_{0,r}( X_{r,0}(x))\Big)^{-1} \Phi (x,\diamond\,\di r) 
\ea
which is understood in the following sense:
\ba
\label{e:inve}
X_{t,0}(x)&=x- \int_0^t \Big(D X_{0,r-}( X_{r-,0}(x))\Big)^{-1} f(x,r)\,\di r \\
&-\int_0^t \Big(D X_{0,r-}( X_{r-,0}(x))\Big)^{-1} F(x,r)\circ \di W_r \\
&+ \int_0^t \int_{\|z\|\leq 1} \Big(\eex^{\psi(\cdot,r,z)}(X_{r-,0}(x)) -X_{r-,0}(x) \Big) \, \tilde N(\di z,\di r) \\
&+\int_0^t\int_{\|z\|\leq 1} \Big(  \eex^{\psi(\cdot,r,z)}(X_{r-,0}(x)) 
-  X_{r-,0}(x) \\
&\hspace{.3\textwidth}-  ( D X_{0,r-}(X_{r-,0}(x)))^{-1} \phi(x ,r,z) \Big) \, \nu(\di z)\,\di r \\
&+\int_0^t \int_{\|z\|> 1} \Big(\eex^{\psi(\cdot,r,z)}(X_{r-,0}(x)) -X_{r-,0}(x) \Big) \, N(\di z,\di r).
\ea
where $\eex^{\psi(\cdot,r,z)}$ is the exponential mapping defined with the help of the solution $w=w(u;y)=w(u; y,r,z)$ of the ODE
\ba
\label{e:MF}
\begin{cases}
\displaystyle
\frac{\di}{\di u}w(u;y)&
=-\Big(D\eex^{u\phi(\cdot, r,z)}(X_{r-}(\cdot))\Big)^{-1}\phi(\eex^{u\phi(\cdot, r,z)}(X_{r-}(\cdot)),r,z)\circ w(u;y) ,\\
w(0;y)&=y,\quad u\in[0,1],
\end{cases}
\ea
i.e.\ $\eex^{\psi(\cdot,r,z)}(y):=\eex^{\psi(r,z)}(y):=w(1;y)$.
\end{thm}
\begin{proof}
For brevity we assume that $\nu(\|z\|>1)=0$ and denote $X_{0,t}=X_t$, and $DX_{0,t}=DX_t$. Recall that $DX_t(x)$ is a right stochastic exponent,
see Section V.9 of Protter \cite{Protter-04} and Section 4 of Fujiwara and Kunita \cite{FujiwaraK-99a} for more detail. It
is well defined and is invertible.

Define the following drift, diffusion and jump coefficients:
\ban
g(y,r)&=-\Big(DX_{r-}(y)\Big)^{-1}f(X_{r-}(y),r),\\
G(y,r)&=-\Big(DX_{r-}(y)\Big)^{-1}F(X_{r-}(y),r),\\
\psi(y,r,z,u)&=-\Big(D \eex^{u\phi(\cdot,r,z)}(X_{r-}(y))\Big)^{-1}\phi(\eex^{u\phi(\cdot,r,z)}(X_{r-}(y)),r,z).
\ean
In particular,
\ban
\psi(y,r,z,0)&=-\Big(D X_{r-}(y)\Big)^{-1}\phi(X_{r-}(y),r,z).
\ean
The functions $g$, $G$ and $\psi$ are predictable
 and with the help of localization we can assume that
they satisfy Assumptions of Section \ref{s:marcus}.
Then equation \eqref{e:MF} has a unique global solution.

Consider the supplementary SDE
\ban
Y_t&=x+\int_0^t g(Y_{r-},r)\,\di r
+\int_0^t G(Y_{r-},r)\circ \di W_r\\ 
& + \int_0^t \int_{\|z\|\leq 1} \Big(\eex^{\psi(r,z)}(Y_{r-}) -Y_{r-} \Big) \, \tilde N(\di z,\di r) \\
&+\int_0^t\int_{\|z\|\leq 1} \Big(  \eex^{\psi(r,z)}(Y_{r-})- Y_{r-}-\Big(D X_{r-}(Y_{r-})\Big)^{-1}\phi(X_{r-}(Y_{r-}),r,z)\Big)\,\nu(\di z)\,\di r,
\ean
where $\eex^{\psi(r,z)}$ is defined in \eqref{e:MF}. 

We show that for each $T>0$ and for any localized solution we have $X_t(Y_t)\equiv x$ on $[0,T]$. 
Let us again consider the one-dimensional case. 
We apply the generalized It\^o formula and show that all the integral terms vanish. Indeed, for the drift term we get
\ban
f(X_{r-}&(Y_{r-}),r,z)+D X_{r-}{Y_{r-}}g(Y_{r-},r,z)\\
&=f(X_{r-}(Y_{r-}),r,z)-D X_{r-}{Y_{r-}}\Big(DX_{r-}(Y_{r-})\Big)^{-1}f(X_{r-}(Y_{r-}),r)\equiv 0.
\ean
The other Lebesgue and It\^o stochastic integrals w.r.t.\ $W$ vanish analogously. To treat the jump terms we 
consider the function $h(u;x):=\eex^{u\phi(\cdot,r,z)}(X_{r-}(x))$
where the mapping $(u,x)\mapsto \eex^{u\phi}(x)$ has been defined in \eqref{e:h1}, \eqref{e:h2}, \eqref{e:h3}, so that
$h(0;x):=X_{r-}(x)$ and $h(1;x):=X_{r}(x)$.
Then
taking into account \eqref{e:MF} we obtain that
\ba
\label{e:hw}
\frac{\di }{\di u} h(u; w(u;y) )
&= \frac{\partial }{\partial u }h(u;w(u;y) )+ \frac{\partial}{\partial x} h(u;w(u;y)) \frac{\di }{\di u }w(u;y)\\
&=\phi(h(u;w(u;y));r,z   )\\
&- \frac{\partial}{\partial x}  h(u;w(u;y)) \Big( \frac{\partial}{\partial x}  h(u;w(u;y))    \Big)^{-1}\cdot 
\phi(h(u;w(u;y));r,z   )\\
&\equiv 0.
\ea
In other words, we have
\ban
\eex^{\phi(r,z)}(X_{r-}(\eex^{\psi(r,z)}(Y_{r-}   ) ) - X_{r-}(Y_{r-})
&=\Big(h(1; w(1;y)) - h(0; w(0,y))\Big)\Big|_{y=Y_{r-}}\\
&=\int_0^1   \frac{\di }{\di u} h(u; w(u;y) )\,\di u\Big|_{y=Y_{r-}}\\
&\equiv 0.
\ean
Furthermore, putting together the compensated terms in the generalized It\^o formula we get
\ban
&\eex^{\phi(\cdot, r,z)}(X_{r-}(Y_{r-}))-X_{r-}(Y_{r-})- \phi(X_{r-}(Y_{r-}),r,z)\\
&+ \eex^{\phi (\cdot,r, z)}\Big(X_{r-}\Big(\eex^{\psi(\cdot,r, z)}(Y_{r-})\Big)\Big)-\eex^{\phi (\cdot,r, z)}(X_{r-}(Y_{r-}))  \\
&- DX_{r-}(Y_{r-} ) \psi(Y_{r-},r,z) 
\equiv 0.
\ean
Hence $Y_t(y)=X_t^{-1}(y)$ for each localized solution $Y$. Since $X$ exists on $[0,T]$, passing to the limit in the localization sequence
we get that the $Y$ is the inverse flow and satisfies the SDE \eqref{e:inve}.
\end{proof}

\begin{thm}[It\^o's formula for the inverse flow w.r.t.\ the first variable]
\label{t:Theta}
Let $\Theta\in C^2(\bR^d,\bR)$. Then the inverse flow $t\mapsto X_{t,0}$ satisfies the canonical SDE
\ba
\label{e:Theta2}
\Theta(X_{t,0}(x))&=\Theta(x)
-\int_0^t  \nabla^T \Theta (X_{r,0}(x)) D X_{r,0}(x) \Phi (x,\diamond\, \di r) \\
&=\Theta(x)-\int_0^t  \nabla^T \Big(\Theta \circ X_{r,0}(x)\Big) \Phi (x,\diamond\,\di r )  ,
\ea
which is understood as follows:
\ba
\label{e:inveIto}
\Theta(X_{t,0}(x))&=\Theta(x)- \int_0^t \nabla^T \Theta(X_{r-,0}(x)) D X_{r-,0}(x) f(x,r)\,\di r \\
&-\int_0^t \nabla^T\Theta(X_{r-,0}(x)) D X_{r-,0}(x)  F(x,r)\circ \di W_r \\
&+ \int_0^t \int_{\|z\|\leq 1} \Big(\Theta(\eex^{-\phi(\cdot, r,z)}(X_{r-,0}(x)))- \Theta(X_{r-,s}(x)) \Big) \, \tilde N(\di z,\di r) \\
&+\int_0^t\int_{\|z\|\leq 1} 
\Big(  \Theta(\eex^{-\phi(\cdot,r,z)}(X_{r-,0}(x))) - \Theta( X_{r-,0}(x))\\
&\hspace{0.2\textwidth}+\nabla^T \Theta(X_{r-,0}(x)) D X_{r-,0}(x) \phi(x ,r,z) \Big) \, \nu(\di z)\,\di r\\
 &+ \int_0^t \int_{\|z\|> 1} \Big(\Theta(\eex^{-\phi(\cdot, r,z)}(X_{r-,0}(x)))- \Theta(X_{r-,0}(x)) \Big) \,  N(\di z,\di r).
\ea
\end{thm}
\begin{proof}
The proof goes along the lines of the proof 
of Theorem 4.4.5 in Kunita \cite{Kunita97}.
For brevity we denote the forward flow by $X_t:=X_{0,t}$ and the inverse flow by $Y_t:=X_{t,0}$, $t\in[0,T]$. We have shown 
that the inverse flow $Y$
satisfies the SDE \eqref{e:inveM}.

First we note that since $X_{t}(Y_t(x))\equiv x$, the gradient matrices  $D X_{t}(x)$ and $D Y_{t}(x)$ satisfy the relation
\ban
D X_{t}(Y_{t}(x))\cdot DY_{t}(x)=\text{Id}
\ean
or equivalently,
\ban
\Big( D X_{t}(Y_{t}(x))\Big)^{-1}=DY_{t}(x).
\ean
Second, taking into account \eqref{e:hw} we get that
\ban
h(u; w(u;x) )\equiv x,\quad u\in[0,1],
\ean
and hence
\ba
\label{e:Dw}
Dh(u;w(u;x))Dw(u;x)=\text{Id} 
\ea
or equivalently,
\ban
\Big(Dh(u;w(u;x))\Big)^{-1}=Dw(u;x),\quad u\in[0,1].
\ean
Thus the equation \eqref{e:MF} for $w$ takes the form
\ban
\begin{cases}
\displaystyle
\frac{\di }{\di u} w(u;x)&=- Dw(u;x)\phi(x,r,z),\quad u\in[0,1],\\
w(0;x)&=x.
\end{cases}
\ean
This is the first order transport equation, its solution is  given by $\eex^{-u\phi(\cdot,r,z)}(x)$ and hence 
\ba
\label{e:ww}
\eex^{\psi(\cdot,r,z)}(\cdot)=\eex^{-\phi(\cdot,r,z)}(\cdot).
\ea
Hence applying the It\^o formula to the equation \eqref{e:inve} and taking into account \eqref{e:Dw} and \eqref{e:ww} yields
\ban
\Theta(Y_t(x))&=x- \int_0^t \nabla^T \Theta(Y_{r-}(x))D Y_{r-}(x)   f(x,r)\,\di r \\
&-\int_0^t \nabla^T  \Theta(Y_{r-}(x))D Y_{r-}(x)  F(x,r)\circ \di W_r \\
&+ \int_0^t \int_{\|z\|\leq 1} \Big(\Theta(\eex^{-\phi(\cdot,r,z)}(Y_{r-})) -\Theta( Y_{r-})  \Big) \, \tilde N(\di z,\di r) \\
&+\int_0^t\int_{\|z\|\leq 1} \Big( \Theta( \eex^{-\phi(\cdot,r,z)}(Y_{r-} ) ) - \Theta( Y_{r-})\\
&\hspace{.2\textwidth}+\nabla^T \Theta(Y_{r-}(x))D Y_{r-}(x)   \phi(x ,r,z) \Big) \, \nu(\di z)\,\di r ,\\
\ean
or equivalently
\ban
\Theta(Y_t(x))&=x- \int_0^t \nabla^T \Big(\Theta\circ Y_{r-}(x)\Big)   f(x,r)\,\di r \\
&-\int_0^t \nabla^T \Big(\Theta\circ Y_{r-}(x)\Big)  F(x,r)\circ \di W_r \\
&+ \int_0^t \int_{\|z\|\leq 1} \Big(\Theta(\eex^{-\phi(\cdot,r,z)}(Y_{r-})) -\Theta( Y_{r-})  \Big) \, \tilde N(\di z,\di r) \\
&+\int_0^t\int_{\|z\|\leq 1} \Big( \Theta( \eex^{-\phi(\cdot,r,z)}(Y_{r-} ) ) - \Theta( Y_{r-})\\
&\hspace{.2\textwidth}+ \nabla^T\Big(\Theta\circ Y_{r-}(x)\Big)   \phi(x ,r,z) \Big) \, \nu(\di z)\,\di r ,\\
\ean
The latter formula can be formally written in the canonical form \eqref{e:Theta2}.
\end{proof}

\section{Proof of Theorem \ref{t:main}\label{s:proof}}

Let $W$ be a $m$-dimensional Brownian motion and $Z$ be $m$-dimensional compensated pure jump L\'evy process.
For simplicity assume that $\nu(\|z\|> 0)=0$.
Consider the linear equation \eqref{e:SPDE}.
To solve it, we consider a $(d+2)$-dimensional system of characteristics Marcus SDEs \eqref{e:phi}, \eqref{e:xi}, \eqref{e:zeta}. 
Note that $\phi$ is a $d$-dimensional process whereas $\xi$ and $\zeta$ are one-dimensional.
Denote $X:=(\phi^1,\dots,\phi^d,\xi,\zeta)=(X^1,\dots,X^d,X^{d+1},X^{d+2})\in \bR^{d+2}$ (the column vector),
and consider the functions
\begin{align}
\notag
f(X)&=-\begin{pmatrix}
      a^1(X^1,\dots, X^d)\\
      \vdots\\
      a^d(X^1,\dots, X^d)\\
      X^{d+1} b(X^1,\dots, X^d)\\
      X^{d+1} c(X^1,\dots, X^d)
     \end{pmatrix},\\
\notag
F(X)&=-\begin{pmatrix}
      A^1_1(X^1,\dots, X^d)&\cdots & A^m_1(X^1,\dots, X^d)\\
      \vdots &\ddots & \vdots\\
      A^1_d(X^1,\dots, X^d)&\cdots & A^m_d(X^1,\dots, X^d) \\
      X^{d+1} B_1(X^1,\dots, X^d)& \cdots&X^{d+1} B_m(X^1,\dots, X^d) \\
       X^{d+1} C_1(X^1,\dots, X^d)& \cdots&X^{d+1} C_m(X^1,\dots, X^d) 
     \end{pmatrix},\\
\label{e:Sigma}     
\Sigma(X)&=- \begin{pmatrix}
      \alpha^1_1(X^1,\dots, X^d)&\cdots & \alpha^m_1(X^1,\dots, X^d)\\
      \vdots &\ddots & \vdots\\
      \alpha^1_d(X^1,\dots, X^d)&\cdots & \alpha^m_d(X^1,\dots, X^d) \\
      X^{d+1} \beta_1(X^1,\dots, X^d)& \cdots&X^{d+1} \beta_m(X^1,\dots, X^d) \\
       X^{d+1} \sigma_1(X^1,\dots, X^d)& \cdots&X^{d+1} \sigma_m(X^1,\dots, X^d) 
     \end{pmatrix}     .
\end{align}
In the matrix form, the system  \eqref{e:phi}, \eqref{e:xi}, \eqref{e:zeta} reads as a canonical equation of the type \eqref{e:SDE}
\ban
X_{0,t}=X_0+\int_0^t \Phi(X_{0,r},\diamond\, \di r)
\ean
with $\phi(X,r,z)=\Sigma(X)z$ (here we allow an abuse of notation). 
There is a unique solution $X=(X_{0,t}(x,\xi_0,\zeta_0))_{t\geq 0}$ which is a $C^2$-flow on $\bR^{d+2}$. 
Denote $Y_t=X_{t,0}=X^{-1}_{0,t}= (\phi^1_{t,0},\dots,\phi^d_{t,0},\xi_{t,0},\zeta_{t,0})$ its inverse flow.

Consider a function
\ban
\Theta(x,\xi_0,\zeta_0)=\xi_0 u_0(x)+\zeta_0
\ean
and define a process
\ban
u(t;x,\xi_0,\zeta_0)&=\Theta(Y_t(x,\xi_0,\zeta_0))\\
&=\xi_{t,0}(x,\xi_0)u_0(\phi_{t,0}(x))+\zeta_{t,0}(x,\xi_0,\zeta_0),\quad x\in\bR^d,\ \xi_0,\zeta_0\in\bR.
\ean
Then by Theorem \ref{t:Theta} we get that
\ban
u(t,x,& \xi_0,\zeta_0)=\Theta(Y_t)
=\Theta(x,\xi_0,\zeta_0)-\int_0^t \nabla^T \Big(\Theta\circ Y_{r}(x,\xi_0,\zeta_0)\Big) \Phi(x,\xi_0,\zeta_0,\diamond\, \di r)\\
&=u_0(0,x,\xi_0,\zeta_0)
+\int_0^t \nabla^T_x u(r-,x,\xi_0,\zeta_0) a(x)\,\di r \\
&+ \xi_0 \int_0^t \partial_{\xi_0} u(r-,x,\xi_0,\zeta_0) b(x)\,\di r
+ \xi_0 \int_0^t \partial_{\zeta_0} u(r-,x,\xi_0,\zeta_0) c(x)\,\di r\\
&+\int_0^t \nabla^T_x u(r-,x, \xi_0,\zeta_0) A(x)\circ \di W_r 
+ \xi_0  \int_0^t \partial_{\xi_0} u(r-,x,\xi_0,\zeta_0) B(x) \circ \di W_r \\
&+ \xi_0  \int_0^t \partial_{\zeta_0} u(r-,x,\xi_0,\zeta_0) C(x) \circ \di W_r \\
&+\int_0^t \int_{\|z\|\leq 1} \Big(\Theta(\eex^{-\Sigma(\cdot) z}(Y_{r-}(x,\xi_0,\zeta_0))) - u(r-,x,\xi_0,\zeta_0)  \Big) \, \tilde N(\di z,\di r) \\
&+\int_0^t\int_{\|z\|\leq 1} \Big( \Theta( \eex^{-\Sigma(\cdot)z }(Y_{r-}(x,\xi_0,\zeta_0) ) ) - u(r-,x,\xi_0,\zeta_0) \\
&\hspace{.2\textwidth} + \nabla^T_x u(r-,x,\xi_0,\zeta_0) \alpha (x) z 
+ \xi_0 \partial_{\xi_0} u(r-,x,\xi_0,\zeta_0)\beta(x) z\\
&\hspace{.2\textwidth}+ \xi_0\partial_{\zeta_0} u(r-,x,\xi_0,\zeta_0) \sigma(x)z
\Big) \, \nu(\di z)\,\di r.
\ean
Let us study the derivatives $\partial_{\xi_0} u$ and $\partial_{\zeta_0} u$.
First we note that the process $\xi$ is found explicitly as an exponential
\ban
\xi_{0,t}&(x,\xi_0)\\
&=\xi_0 \exp\Big(-\int_0^t  b( \phi_{0,r}(x) )\,\di r
- \int_0^t  B(\phi_{0,r}(x))\circ \di W_r
-  \int_0^t\beta(\phi_{0,r}(x))\diamond \di Z_r\Big).
\ean
The derivative equals to
\ban
\partial_{\xi_0}\xi_{0,t}&(x,\xi_0)\\
&= \exp\Big(-\int_0^t  b( \phi_{0,r}(x) )\,\di r
- \int_0^t  B(\phi_{0,r}(x))\circ \di W_r
-  \int_0^t\beta(\phi_{0,r}(x))\diamond \di Z_r\Big)
\ean
and it satisfies \eqref{e:xi} with the initial values $(x,\xi_0)=(x,1)$.
By the formula of the derivative of the inverse function we get
\ba
\label{e:dy}
\partial_{\xi_0} \xi_{t,0}(x,\xi_0)&=\Big(\partial_{\xi_0} \xi_{0,t}(\phi_{t,0}(x),\xi_{t,0}(x,\xi_0))   \Big)^{-1}\\
&=\Big(\xi_{0,t}(\phi_{t,0}(x),1)   \Big)^{-1}=
\xi_{t,0}(x,1).
\ea 
Analogously,
\ba
\label{e:dz}
&\partial_{\xi_0}\zeta_{t,0}(x,\xi_0,\zeta_0)=\zeta_{t,0}(x,1,0),\\
&\partial_{\zeta_0}\zeta_{t,0}(x,\xi_0,\zeta_0)=1.
\ea
Thus taking into account \eqref{e:dy} and \eqref{e:dz} we can write
\ban
\partial_{\xi_0} u(r,x,\xi_0,\zeta_0)&=\partial_{\xi_0}\Big( \Theta(\phi_{t,0}(x),\xi_{t,0}(x,\xi_0),\zeta_{t,0}(x,\xi_0,\zeta_0))\Big)\\
&=u_0(\phi_{t,0}(x))\partial_{\xi_0} \xi_{t,0}(x,\xi_0) + \partial_{\xi_0} \zeta_{t,0}(x,\xi_0,\zeta_0)\\
&=u_0(\phi_{t,0}(x))  \xi_{t,0}(x,1) + \zeta_{t,0}(x,1,0)=u(r,x,1,0)
\ean
and 
\ban
\partial_{\zeta_0} u(r,x,\xi_0,\zeta_0)&=\partial_{\zeta_0} \Big( \Theta(\phi_{t,0}(x),\xi_{t,0}(x,\xi_0),\zeta_{t,0}(x,\xi_0,\zeta_0))\Big)\\
&=\partial_{\zeta_0} \zeta_{t,0}(x,\xi_0,\zeta_0)=1.
\ean
Inspecting the structure of the matrix function $\Sigma$ in \eqref{e:Sigma} we get that the mapping $\eex^{\Sigma(\cdot)z}\colon \bR^{d}\to\bR^{d+2}$ 
has the following form:
\ban
\eex^{\Sigma(\cdot)z}
\begin{pmatrix}
x\\ \xi_0\\\zeta_0
\end{pmatrix}
=\begin{pmatrix}
\eex^{-\alpha(\cdot)z}(x)\\ 
\displaystyle \xi_0\exp\Big(-\int_0^1 \beta(\eex^{-r\alpha(\cdot)z}(x))z\,\di r    \Big) \\
\displaystyle \zeta_0 - \xi_0\int_0^1 \exp\Big(-\int_0^s \beta(\eex^{-r\alpha(\cdot)z}(x))z\,\di r    \Big) \sigma(\eex^{-r\alpha(\cdot)z}(x) ) \,\di s
\end{pmatrix}.
\ean
Recalling \eqref{e:ww}, namely that
\ban
\eex^{\cQ z}(\Theta(\cdot))=\Theta(\eex^{-\Sigma(\cdot)z}(\cdot)),
\ean
we get the equality
\ban
&u(t,x,1,0)=u_0(x)
+\int_0^t \nabla^T u(r-,x,1,0) a(x)\,\di r 
+ \int_0^t u(r-,x,1,0) b(x)\,\di r
+ c(x)t\\
&+ \int_0^t \nabla^T u(r-,x,1,0) A(x)\circ \di W_r 
+ \int_0^t u(r-,x,1,0) B(x) \circ \di W_r 
+   C(x) W_t \\
&+\int_0^t \int_{\|z\|\leq 1} \Big(\eex^{\cQ z}(u(r-,x,1,0)) - u(r-,x,1,0)  \Big) \, \tilde N(\di z,\di r) \\
&+\int_0^t\int_{\|z\|\leq 1} \Big(\eex^{\cQ z}(u(r-,x,1,0))  - u(r-,x,1,0)
- \cQ u(r-,x,1,0) z
\Big) \, \nu(\di z)\,\di r.
\ean
This means that $u(t,x)=u(t,x,1,0)$ is the solution of \eqref{e:SPDE-2}.

To show uniqueness, 
let us first assume that $b(x)=0$, $B(x)=0$, $\beta(x)=0$, and $c(x)=0$, $C(x)=0$, $\sigma(x)=0$.
In this case, the solution defined by the characteristics has the form
\ba
\label{e:u}
u(t,x)=u_0(\phi_{t,0}(x)).
\ea
Let $v$ be another semimartingale solution of the form 
\ban
v(t,x)=u_0(x)&+\int_0^t f(x,r)\,\di r + \int_0^t F(x,r)\,\di W_r
+ \int_0^t\int_{\|z\|\leq 1}   \phi(x, r, z) \, \tilde N(\di z,\di r)
\ean
with some $f,F$ and $\phi$.
Then Theorem \ref{t:Ito2} yields that
\ban
&v(t,\phi_{0,t}(x))=
u_0(x)+
\int_0^t f(\phi_{0,r-}(x),r)\,\di r + \int_0^t F(\phi_{0,r-}(x),r)\,\di W_r\\
&+ \int_0^t\int_{\|z\|\leq 1}   \phi(\phi_{0,r-}(x), r, z) \, \tilde N(\di z,\di r) \\
&- \int_0^t \partial_x v(r-,\phi_{0,r-}(x)) a(\phi_{0,r-}(x))\,\di r 
-  \int_0^t  \partial_x v (r-,\phi_{0,r-}(x)) A(\phi_{0,r-}(x)) \circ \di W_r \\
&+\int_0^t \int_{\|z\|\leq 1} \Big(v(r-,\eex^{-\alpha}(\phi_{0,r-})) + \phi(\eex^{-\alpha}(\phi_{0,r-})),r,z)\\
&\hspace{.4\textwidth}- v(r-,\phi_{0,r-})- \phi(r-,\phi_{0,r-},r,z) \Big) \, \tilde N(\di z,\di r) \\
&+\int_0^t\int_{\|z\|\leq 1} \Big(v(r-,\eex^{-\alpha}(\phi_{0,r-})) + \phi(\eex^{-\alpha}(\phi_{0,r-})),r,z)\\
&\hspace{.5\textwidth}+ \partial_x v(r-,\phi_{0,r-}) \alpha(\phi_{0,r-})\Big) \, \nu(\di z)\,\di r,\\
&-\int_0^t\int_{\|z\|\leq 1} \Big( v(r-,\phi_{0,r-})+ \phi(r-,\phi_{0,r-},r,z)\Big) \, \nu(\di z)\,\di r,\\
\ean
where we know that $\phi$ is given by
\ban
\phi(x,r,z) = \eex^{\cQ z}(v(r-,x)) - v(r-,x) .
\ean
Let us take a closer look at the jump terms. On the one hand we know that
\ban
v(r,\eex^{\alpha(\cdot)z}(\phi_{0,r-}(x)))=\eex^{\cQ z}v(r,\phi_{0,r-}(x)).
\ean
On the other hand, inverting the sign of the jump size $z$ is equivalent to the reversion of the 
fictitious time in the Marcus ODE for $\ex^{\alpha(\cdot)z}$. Hence we obtain that
\ban
\eex^{\cQ z} (v(r,\eex^{-\alpha(\cdot)z}(\phi_{0,r-}(x))))=v(r,\phi_{0,r-}(x)).
\ean
We also see that
\ban
\partial_x v(r-,\phi_{0,r-}(x)) \alpha(\phi_{0,r-}(x)) = \cQ v(r-,\phi_{0,r-}(x)),
\ean
and thus we get
\ban
&v(t,\phi_{0,t}(x))=u_0(x)
+ \int_0^t \partial_t v(r-,\phi_{0,r-}(x)) \diamond \di \phi_{0,r}(x)\\
&- \int_0^t \partial_x v(r-,\phi_{0,r-}(x)) a(\phi_{0,r-}(x))\,\di r 
-  \int_0^t  \partial_x v (r-,\phi_{0,r-}(x)) A(\phi_{0,r-}(x)) \circ \di W_r \\
&-\int_0^t \int_{\|z\|\leq 1} \Big(\eex^{\cQ z}(v(r-,\phi_{0,r-}(x)) - v(r-,\phi_{0,r-}(x))  \Big) \, \tilde N(\di z,\di r) \\
&-\int_0^t\int_{\|z\|\leq 1} \Big(\eex^{\cQ z}(v(r-,\phi_{0,r-}(x)))  - v(r-,\phi_{0,r-}(x))
- \cQ v(r-,\phi_{0,r-}(x)) z
\Big) \, \nu(\di z)\,\di r\\
&=u_0(x)
\ean
and hence $v(t,x)$ coincides with the solution $u$ given by \eqref{e:u}.

In the presence of linear terms $b$, $B$ and $\beta$,
the process $u(t,\phi_{t,0}(x))$ given by the characteristics solution has the form
\ban
u&(t,\phi_{t,0}(x))=u_0(\phi_{t,0}(x))\times\\
&\times\exp\Big(\int_0^t  b( \phi_{0,r}(x) )\,\di r
+\int_0^t  B(\phi_{0,r}(x))\circ \di W_r
+ \int_0^t\beta(\phi_{0,r}(x))\diamond \di Z_r\Big).
\ean
and the difference $d(t,x):=v(t,\phi_{0,t}(x))- u(t,\phi_{t,0}(x))$ satisfies the linear equation
\ban
d(t,x)=\int_0^t d(r-,x) b(\phi_{r-}(x))\,\di r &+ \int_0^t d(r-,x) B(\phi_{r-}(x))\circ \di W_r \\
&+ \int_0^t d(r-,x) \beta(\phi_{r-}(x))\diamond \di Z_r,
\ean
and thus $d\equiv 0$. The same relation holds for the difference of the non-homogeneous equations.

\section{Proof of Theorem \ref{t:Ito}\label{s:gIf}}
 
To simplify the notation, we assume that $X$ and $Y$, as well as $W$ and $Z$ are one-dimensional processes, i.e.\ $d=m=1$. We also assume that 
$\nu([-1,1]^c)=0$. Adding the large jumps is straightforward.

Let us write the SDEs for $X$ and $Y$ in the It\^o form:
\ban
X_{t}(x)
&=x+\int_0^t f(X_{r-}(x),r)\,\di r +   \int_0^t F(X_{r-}(x),r)\,\di W_r\\
&+ \frac12   \int_0^t F'(X_{r-}(x),r)  F(X_{r-}(x),r) \,\di r\\
&+ \int_0^t\int_{|z|\leq 1} \Big(\eex^{\phi(\cdot,r, z)}(X_{r-}(x)) -X_{r-}(x)  \Big) \, \tilde N(\di z,\di r) \\
&+ \int_0^t\int_{|z|\leq 1} \Big(\eex^{\phi(\cdot,r ,z)}(X_{r-}(x)) -X_{r-}(x) - \phi(X_{r-}(x),r, z) \Big) \, \nu(\di z)\,\di r \\
Y_t&=y+\int_0^t g(Y_{r-},r)\,\di r
+  \int_0^t G(Y_{r-},r)\,\di W_r
+ \frac12  \int_0^t G'(Y_{r-},r)  G(Y_{r-},r)\,\di r\\
&+ \int_0^t\int_{|z|\leq 1} \Big(\eex^{\psi(\cdot,r, z)}(Y_{r-}) -Y_{r-}  \Big) \, \tilde N(\di z,\di r) \\
&+ \int_0^t\int_{|z|\leq 1} \Big(\eex^{\psi(\cdot,r, z)}(Y_{r-}) -Y_{r-} -  \psi(Y_{r-},r, z) \Big) \, \nu(\di z)\,\di r.
\ean
First we perform localization of the semimartingales $X(x)$ and $Y$.
We assume that the semimartingales $\Phi$ and $\Psi$ satisfy the assumptions in Section \ref{s:marcus} 
so that $X$ and $Y$ are well defined and sufficiently smooth.
Since the jumps of $Y$ are assumed to be bounded, for any initial value $y$ we can localise $Y$ so that the stopped process belongs to a certain ball.
Similarly, we can stop $X$ such that it is also bounded with all its derivatives up to order 2 uniformly for all values $x$ in the ball defined above. 
Therefore from now on we will work with the stopped semimartingales and also assume that all the coefficients $f$, $F$, $\phi$, $g$, $G$, $\psi$ 
as well as their derivatives have compact support.

For the proof of the generalized It\^o formula we apply the method by Carmona and Nualart \cite{carmona1990nonlinear}, Theorem III.3.3.

Consider a sequence of mollifiers $h_n\in C^\infty_c(\bR,\bR)$ given by 
$h_n(x)=n h(n x)$,
where $h\in C^\infty_c(\bR^d,\bR)$ supported on a unit ball $|x|\leq 1$, $h(x)\geq 0$, and such that $\int_{\bR} h(x)\,\di x=1$.
For the smoothing properties of mollifiers see, e.g.\ Evans \cite[Appendix C.4]{evans2002partial}.

Then for each $x\in\bR$, the classical
It\^o formula applied to the semimartingale $Y$ yields
\ba
\label{e:hY}
&h_n(Y_t-x)=h_n(y-x)+ \int_0^t h'_n(Y_{r-}-x) g(Y_{r-},r)\,\di r \\
&+ \int_0^t h'_n(Y_{r-}-x) G (Y_{r-},r)\,\di W_r\\
&+ \frac12 \int_0^t h'_n(Y_{r-} -x)  G'(Y_{r-},r)  G(Y_{r-},r)\,\di r\\  
&+\frac12  \int_0^t h''_n(Y_{r-} -x) G^2(Y_{r-} ,r)\,\di r\\
&+ \int_0^t\int_{|z|\leq 1} h'_n(Y_{r-}-x)\Big( \eex^{\psi(\cdot,r, z)}(Y_{r-})- Y_{r-}-\psi(Y_{r-},r,z) \Big) \, \nu(\di z)\,\di r \\
&+ \int_0^t\int_{|z|\leq 1} \Big[ h_n\Big(\eex^{\psi(\cdot,r, z)}(Y_{r-}) -x\Big)- h_n (Y_{r-} -x) \Big] \, \tilde N(\di z,\di r) \\
&+\int_0^t\int_{|z|\leq 1} \Big[ h_n\Big(\eex^{\psi(\cdot,r, z)}(Y_{r-}) -x\Big)- h_n (Y_{r-} -x) \\
&\hspace{.3\textwidth}- h'_n(Y_{r-} -x) 
\Big(\eex^{\psi(\cdot,r, z)}(Y_{r-}) -Y_{r-}\Big)  \Big] \, \nu(\di z)\,\di r.
\ea
Next we apply the It\^o product formula to $X(x)h(Y-x)$ to get
\ban
X_t(x) &h_n(Y_t-x)= x h_n(y-x)\\
I_1&=
\begin{cases}
\displaystyle +\int_0^t h_n(Y_{r-} -x) f(X_{r-}(x),r)\,\di r \\
\end{cases}\\
I_2&=
\begin{cases}
\displaystyle + \frac12   \int_0^t h_n(Y_{r-} -x) F'(X_{r-}(x),r)  F(X_{r-}(x),r) \,\di r \\
\end{cases}\\
I_3&=
\begin{cases}
\displaystyle 
+ \int_0^t \int_{|z|\leq 1} h_n(Y_{r-}-x)\Big(\eex^{\phi (\cdot,r, z)}(X_{r-}(x)) -X_{r-}(x) \\
\hspace{.5\textwidth} \displaystyle - \phi(X_{r-}(x),r, z) \Big) \, \nu(\di z)\,\di r
\end{cases}\\
I_4&=
\begin{cases}
\displaystyle +   \int_0^t h_n(Y_{r-} -x) F (X_{r-}(x),r)\,\di W_r \\
\end{cases}
\ean
\ba
\label{e:product}
J_1&=
\begin{cases}
\displaystyle + \int_0^t X_{r-}(x) h'_n(Y_r-x) g(Y_{r-},r)\,\di r \\
\end{cases}\\
J_2&=
\begin{cases}
\displaystyle + \frac12 \int_0^t  X_{r-}(x)
h'_n(Y_r-x)   G'(Y_{r-},r)  G(Y_{r-},r)\,\di r   \\
\end{cases}\\
J_3&=
\begin{cases}
\displaystyle +\frac12 \int_0^t X_{r-}(x)h''_n(Y_r-x) G^2(Y_{r-},r)\,\di r\\
\end{cases}\\
J_4&=
\begin{cases}
\displaystyle +\int_0^t\int_{|z|\leq 1} X_{r-}(x)\Big[ h_n\Big(\eex^{\psi(\cdot,r, z)}(Y_{r-}) -x\Big)- h_n (Y_{r-} -x) \\
\displaystyle \hspace{.3\textwidth}- h'_n(Y_{r-} -x) 
\Big(\eex^{\psi(\cdot,r, z)}(Y_{r-}) -Y_{r-}\Big)  \Big] \, \nu(\di z)\,\di r \\
\end{cases}\\
J_5&=
\begin{cases}
\displaystyle +\int_0^t\int_{|z|\leq 1} X_{r-}(x)
h'_n(Y_{r-}-x)\Big( \eex^{\psi(\cdot,r, z)}(Y_{r-})\\
\hspace{.5\textwidth}- Y_{r-}-\psi(Y_{r-},r,z) \Big)
\, \nu(\di z)\,\di r \\
\end{cases}\\
J_6&=
\begin{cases}
\displaystyle + \int_0^t X_{r-}(x)h'_n(Y_r-x) G (Y_{r-},r)\,\di W_r\\
\end{cases}\\
J_7&=
\begin{cases}
\displaystyle + \frac12\int_0^t  F(X_{r-}(x),r) h'_n(Y_r-x) G (Y_{r-},r)\,\di r\\
\end{cases}\\
J_8&=
\begin{cases}
\displaystyle +  
\int_0^t  \int_{|z|\leq 1} \Big[\eex^{\phi(\cdot,r, z)}(X_{r-}(x)) -X_{r-}(x)\Big]\times\\
\hspace{.25\textwidth}\displaystyle \times\Big[ h_n\big(\eex^{\psi(\cdot,r, z)}(Y_{r-}) -x\big)-  h_n(Y_{r-}-x)\Big]\, \nu(\di z)\,\di r\\
\end{cases}\\
K_1&=
\begin{cases}
\displaystyle   + \int_0^t  \int_{|z|\leq 1}  \Big[ \eex^{\phi(\cdot,r, z)}(X_{r-}(x)) h_n\Big(\eex^{\psi(\cdot,r, z)}(Y_{r-}) -x\Big)\\
\displaystyle \hspace{.45\textwidth}- X_{r-}(x)h_n (Y_{r-} -x) \Big] \, \tilde N(\di z,\di r).
\end{cases}
\ea
We decompose the term $K_1$ further into the sum
\ban
I_5&=
\begin{cases}
\displaystyle   \int_0^t  \int_{|z|\leq 1} \Big(\eex^{\phi(\cdot,r, z)}(X_{r-}(x)) - X_{r-}(x)\Big)h_n (Y_{r-} -x) 
 \, \tilde N(\di z,\di r)\\
 \end{cases}\\
J_9&=
\begin{cases}
+\displaystyle \int_0^t  \int_{|z|\leq 1}  \eex^{\phi(\cdot,r, z)}(X_{r-}(x))\times\\
 \displaystyle\hspace{.25\textwidth}  \times\Big[ h_n\Big(\eex^{\psi(\cdot,r, z)}(Y_{r-}) -x\Big)
- h_n (Y_{r-} -x) \Big] \, \tilde N(\di z,\di r).
\end{cases}
\ean
All the (stochastic) integrals exist due to the integrability assumptions on the functions $\phi$ and $\psi$ and the properties of the exponential mappings.

The proof of the generalized It\^o formula will consist in integration of the equality \eqref{e:product} w.r.t.\ $x$ and passing to the limit
as $n\to\infty$.

We distinguish between the Lebesgue integrals w.r.t.\ $\di r$, the It\^o integrals w.r.t.\ $\di W$ and 
the compensated Poissonian random measure $\tilde N$, and the terms containing the derivatives $h'_n$ and $h''_n$.
 
We start with the terms coming from the integral $\int_0^t h_n(Y_{r-}-x)\,\di X_r$ which will give us the first 
integral in \eqref{e:ito-gen}.

\medskip
\noindent
\textbf{Initial and end points.} It follows from the properties of the mollifiers and the continuity of $x\mapsto X_t(x)$ that
\ban
\lim_{n\to\infty}\int_\bR X_t(x) &h_n(Y_t-x)\,\di x=X_t(Y_t) \quad\text{and}\quad
\lim_{n\to\infty}\int_\bR x &h_n(y-x)\,\di x=y. 
\ean

\medskip
\noindent
\textbf{Terms $I_1$ and $I_2$.} We start with the Lebesgue integrals coming from the drift part and 
the noise-induced drift appearing in the Stratonovich 
integrals w.r.t.\ $W$. For each $\omega\in \Omega$, 
\ban
\int_0^t \int_{\bR}|h_n(Y_{r-}-x)  f(X_{r-}(x),r)|\,\di x\,\di r  
&\leq \|f\|\int_0^t \int_{\bR}h_n(Y_{r-}-x)  \,\di x\,\di r\\
&=T\cdot \|f\|<\infty
\ean
and 
the Fubini theorem yields 
\ban
\int_{\bR}\int_0^t h_n(Y_{r-}-x) & f(X_{r-}(x),r)\,\di r  \,\di x
=\int_0^t  \int_{\bR^d}h_n(Y_{r-}-x)  f(X_{r-}(x),r) \,\di x\,\di r\\
&=\int_0^t  \int_{\|x\|\leq 1/n}h_n(x)  f(X_{r-}(Y_{r-}-x),r) \,\di x\,\di r.
\ean
For each $r\in[0,T]$, the function $y\mapsto f(X_{r-}(y),r)$ is continuous and for any $K>0$ by Theorem 6(iii) in Evans 
\cite[Appendix C.4]{evans2002partial}

\ban
\lim_{n\to\infty}\sup_{|y|\leq K}\Big|\int_{|x|\leq 1/n} h_n(x)  f(X_{r-}(y-x),r) \,\di x - f(X_{r-}(y),r)\Big|=0.
\ean
Since for $r\in[0,T]$
\ban
 \Big|\int_{|x|\leq 1/n} h_n(x)  f(X_{r-}(Y_{r-}-x),r) \,\di x \Big|\leq \|f\|
\ean
the Lebesgue theorem implies that
\ban
\int_0^t\int_{\bR} &h_n(Y_{r-}-x)  f(X_{r-}(x),r) \,\di x\,\di r  \to \int_0^t f(X_{r-}(Y_{r-}),r)\,\di r .
\ean
Analogously we get the convergence of the term $I_2$.

\medskip
\noindent
\textbf{Term $I_3$.} 
Consider the function
\ban
H^{I_3}(x,r)=\int_{|z|\leq 1} \Big(\eex^{\phi (\cdot,r, z)}(x) -x - \phi(x,r, z) \Big) \, \nu(\di z). \\
\ean
Since
\ban
|\eex^{\phi (\cdot,r, z)}(x) -x - \phi(x,r, z)|\leq K_0(z)L_0(z) \ex^{K_0(z)}(1+|x|)
\ean
we get that $x\mapsto H^{I_3}(x,r)$ is well-defined and continuous. Recalling that $X$ and $Y$ are assumed to be bounded, 
the argument of the previous step applies and
\ban
\int_0^t  & h_n(Y_{r-}-x)\int_{|z|\leq 1}\Big(\eex^{\phi (\cdot,r, z)}(X_{r-}(x)) -X_{r-}(x) - \phi(X_{r-}(x),r, z) \Big) \, \nu(\di z)\,\di r\\
&\to \int_0^t \int_{|z|\leq 1}\Big(\eex^{\phi (\cdot,r, z)}(X_{r-}(Y_{r-})) -X_{r-}(Y_{r-}) - \phi(X_{r-}(Y_{r-}),r, z) \Big) \, \nu(\di z)\,\di r.
\ean

\medskip
\noindent
\textbf{Term $I_4$.} 
By Fubini's theorem for stochastic integrals, see Protter \cite[Theorem 6.64]{Protter-04}, for each $|y|\leq K$ 
\ban
\int_{\bR}\int_0^t h_n(y-x) F(X_{r-}(x),r)\,\di W_r\,\di x= \int_0^t \int_{\bR}h_n(y-x) F(X_{r-}(x),r)\,\di x\,\di W_r.
\ean
By the properties of mollifiers, see Evans \cite[Appendix C.4, Theorem 6]{evans2002partial}, for each $r\in[0,T]$ and $|y|\leq K$
\ban
\Big|\int_{\bR}h_n(y-x) F(X_{r-}(x),r)\,\di x  - F(X_{r-}(y),r)\Big|\leq 2\|F\|
\ean
and by the It\^o isometry and the Lebesgue theorem
\ban
\E\Big| &\int_0^t \int_{\bR}h_n(Y_{r-}-x) F(X_{r-}(x),r)\,\di x\,\di W_r - 
\int_0^t F(X_{r-}(Y_{r-}),r)\,\di W_r \Big|^2\\
&=\int_0^t \E \Big|\int_{\bR}h_n(Y_{r-}-x) F(X_{r-}(x),r)\,\di x  - F(X_{r-}(Y_{r-}),r)\Big|^2\,\di r \to 0.
\ean

\medskip
\noindent
\textbf{Term $I_5$.}
The jump term $I_5$ is estimated analogously with the help of the It\^o isometry for stochastic integrals w.r.t.\ a compensated 
Poisson random measure.

\medskip
\noindent
\textbf{Terms $J_1$, $J_2$ and $J_7$.} Consider the Lebesgue integral $J_1$. We apply Fubini's theorem for each $\omega$, integrate by parts, 
use that $X$ is a $C^1$-diffeomorphism, and apply Lebesgue's theorem:
\ban
\int_{\bR}\int_0^t X_{r-}(x)h'_n(Y_{r-}-x) &g (Y_{r-},r)\,\di r \,\di x\\
&=\int_0^t\Big[ \int_{\bR}  X_{r-}(x) h'_n(Y_{r-}-x)\,\di x\Big]\, g(Y_{r-},r)\di r \\
&=\int_0^t \Big[\int_{\bR} X'_{r-}(x)h_n(Y_{r-}-x)\,\di x\Big] g(Y_{r-},r)\,\di r \\
&\to  \int_0^t X'_{r-}(Y_{r-})  g(Y_{r-},r)\,\di r .
\ean
The terms $J_2$ and $J_7$ are treated analogously.

\noindent
\textbf{Term $J_3$.} Analogously, using that $X$ is a $C^2$-diffeomorphism, and applying integration by parts twice we get 
\ban
\int_\bR\int_0^t X_{r-}(x)h''_n(Y_{r-}-x)& G^2(Y_{r-},r)\,\di r\,\di x\\
&=\int_0^t \Big[ \int_\bR  X_{r-}(x)h''_n(Y_{r-}-x)  \,\di x\Big]\,G^2(Y_{r-},r)\,\di r\\
&=\int_0^t \Big[ \int_\bR X''_{r-}(x)h_n(Y_{r-}-x)  \,\di x\Big]\,G^2(Y_{r-},r)\,\di r\\
&\to \int_0^t X''_{r-}(Y_{r-}) G^2(Y_{r-},r)\,\di r.
\ean

\noindent
\textbf{Terms $J_4+J_5+J_8$.} First we note that due to the assumptions, the sum of the terms $J_4$, $J_5$ and $J_8$ is well defined.
Hence the sum $J_4+J_5+J_8$ simplifies to
\ban
J_4+J_5+J_8
=\int_0^t\int_{|z|\leq 1} &\Big[\eex^{\phi(\cdot,r, z)}(X_{r-}(x)) 
\Big( h_n\big(\eex^{\psi(\cdot,r, z)}(Y_{r-}) -x\big)-  h_n(Y_{r-}-x)\Big)\\
&-X_{r-}(x) \Big(\eex^{\psi(\cdot,r, z)}(Y_{r-}) -Y_{r-}\Big) h'_n(Y_{r-} -x) \Big] \, \nu(\di z)\,\di r
\ean
and by Fubini's theorem, integration by parts and by the dominated convergence theorem we get
\ban
\lim_{n\to\infty}&\int_\bR (J_4+J_5+J_8)\,\di x\\
&= \int_0^t \int_{|z|\leq 1}  
\Big[ \eex^{\phi (\cdot,r, z)}\Big(X_{r-}\Big(\eex^{\psi(\cdot,r, z)}(Y_{r-})\Big)\Big) - \eex^{\phi(\cdot,r, z)}(X_{r-}(Y_{r-}))\\
&\hspace{0.45\textwidth}- X'_{r-}(Y_{r-}) \psi(Y_{r-},r, z) \Big]
\,  \nu(\di z)\,\di r.
\ean

\medskip
\noindent
\textbf{Term $J_6$.} 
The term $J_6$ is treated with the help of integration by parts analogously to the term $I_4$.

\medskip
\noindent
\textbf{Term $J_9$.} 
The term $J_9$ is treated analogously to the term $I_5$.

\section{Proof of Theorem \ref{t:Ito2}\label{s:ito2proof}}

The proof of this Theorem is analogous. After localization, and application of a mollifier $h_n$ to $Y$ we obtain again the formula \eqref{e:hY}.
The product formula for $\Phi(x,t)h_n(Y_t-x)$ takes the form
\ban
&\Phi(x,t) h_n(Y_t-x)= \Phi(x,0)h_n(y-x)
+\int_0^t h_n(Y_{r-} -x) f(X_{r-}(x),r)\,\di r \\
&  +   \int_0^t h_n(Y_{r-} -x) F (X_{r-}(x),r)\,\di W_r 
+ \int_0^t \Phi(x,r-) h'_n(Y_{r-}-x) g(Y_{r-},r)\,\di r \\
& + \frac12 \int_0^t  \Phi(x,r-) h'_n(Y_{r-}-x)   G'(Y_{r-},r)  G(Y_{r-},r)\,\di r   \\
&  +\frac12 \int_0^t \Phi(x,r-)h''_n(Y_{r-}-x) G^2(Y_{r-},r)\,\di r\\
&  +\int_0^t\int_{|z|\leq 1} \Phi(x,r-)\Big[ h_n\Big(\eex^{\psi(\cdot,r, z)}(Y_{r-}) -x\Big)- h_n (Y_{r-} -x) \\
&\hspace{0.35\textwidth}- h'_n(Y_{r-} -x) 
\Big(\eex^{\psi(\cdot,r, z)}(Y_{r-}) -Y_{r-}\Big)  \Big] \, \nu(\di z)\,\di r \\
&   +\int_0^t\int_{|z|\leq 1} \Phi(x,r-) h'_n(Y_{r-}-x)\Big( \eex^{\psi(\cdot,r, z)}(Y_{r-})- Y_{r-}-\psi(Y_{r-},r,z) \Big)
\, \nu(\di z)\,\di r \\
&  + \int_0^t \Phi(x,r-)h'_n(Y_r-x) G (Y_{r-},r)\,\di W_r\\
&  +\frac 12 \int_0^t  F(X_{r-}(x),r) h'_n(Y_r-x) G (Y_{r-},r)\,\di r\\
\ean
\ba
\label{e:product2}
&  +  \int_0^t  \int_{|z|\leq 1} \phi(x,r,z)
\Big( h_n\big(\eex^{\psi(\cdot,r, z)}(Y_{r-}) -x\big)-  h_n(Y_{r-}-x)\Big)\, \nu(\di z)\,\di r\\
&  + \int_0^t  \int_{|z|\leq 1}  \Big[ \Big( \Phi(x,r-)+\phi(x,r,z)\Big) h_n\Big(\eex^{\psi(\cdot,r, z)}(Y_{r-}) -x\Big)\\
&\hspace{0.5\textwidth}-  \Phi(x,r-) h_n (Y_{r-} -x) \Big] \, \tilde N(\di z,\di r).
\ea
Integrating w.r.t.\ $x$ and passing to the limit and $n\to\infty$ get the formula
\ban
&\Phi(Y_t,t)= \Phi(y,0)
+\int_0^t f( Y_{r-},r)\,\di r  +   \int_0^t F ( Y_{r-},r)\,\di W_r\\
&+ \int_0^t \nabla^T \Phi(Y_{r-},r-) g(Y_{r-},r)\,\di r  +\frac12 \int_0^t \nabla^T \Phi(Y_{r-} ,r-)  G^2(Y_{r-},r)\,\di r\\
&+\frac12 \int_0^t  \nabla^T\nabla \Phi(Y_{r-},r-)   G'(Y_{r-},r)  G(Y_{r-},r)\,\di r   \\
&  +\int_0^t\int_{|z|\leq 1} \Big( \Phi(\eex^{\psi(\cdot,r, z)}(Y_{r-}),r-))- \Phi(Y_{r-},r-)\\
&\hspace{0.5\textwidth}-\nabla^T \Phi(Y_{r-},r-)  \psi(Y_{r-},r,z)
\Big) \, \nu(\di z)\,\di r \\
&  + \int_0^t  \nabla^T \Phi(Y_{r-},r-)   G (Y_{r-},r)\,\di W_r 
 + \frac12 \int_0^t  \nabla^T F(Y_{r-},r) G (Y_{r-},r)\,\di r\\
& +  \int_0^t  \int_{|z|\leq 1} \Big(\phi(\eex^{\psi(\cdot,r, z)}(Y_{r-}),r,z)
-  \phi(Y_{r-},r,z)\Big)\, \nu(\di z)\,\di r\\
&  + \int_0^t  \int_{|z|\leq 1}  \Big[  \Phi(\eex^{\psi(\cdot,r, z)}(Y_{r-}),r-)+\phi(\eex^{\psi(\cdot,r, z)}(Y_{r-}),r,z)
-  \Phi(Y_{r-},r-) \Big] \, \tilde N(\di z,\di r).
\ean
which can be transformed to \eqref{e:Phinabla}, \eqref{e:Phinabla1}, \eqref{e:Phinabla2}.
%
%
%

\end{document}